\documentclass[12pt,a4paper]{amsart}
\usepackage{amsmath,amsthm,amssymb,latexsym,a4wide,epic,eepic,bbm}

\newtheorem{theorem}{Theorem}
\newtheorem{lemma}[theorem]{Lemma}
\newtheorem{corollary}[theorem]{Corollary}
\newtheorem{proposition}[theorem]{Proposition}

\newtheorem{example}[theorem]{Example}
\newtheorem{remark}[theorem]{Remark}

\makeatletter
\newcommand*{\drightleftarrow}[2]{\mathrel{
  \settowidth{\@tempdima}{$\scriptstyle#1$}
  \settowidth{\@tempdimb}{$\scriptstyle#2$}
  \ifdim\@tempdimb>\@tempdima \@tempdima=\@tempdimb\fi
  \mathop{\vcenter{
    \offinterlineskip\ialign{\hbox to\dimexpr\@tempdima+2em{##}\cr
    \rightarrowfill\cr\noalign{\kern.3ex}
    \leftarrowfill\cr}}}\limits^{\!#1}_{\!#2}}}
    
\newcommand*{\drightarrow}[2]{\mathrel{
  \settowidth{\@tempdima}{$\scriptstyle#1$}
  \settowidth{\@tempdimb}{$\scriptstyle#2$}
  \ifdim\@tempdimb>\@tempdima \@tempdima=\@tempdimb\fi
  \mathop{\vcenter{
    \offinterlineskip\ialign{\hbox to\dimexpr\@tempdima+2em{##}\cr
    \rightarrowfill\cr\noalign{\kern.3ex}
    \rightarrowfill\cr}}}\limits^{\!#1}_{\!#2}}}
\makeatother

\usepackage{tikz}
\usetikzlibrary{matrix,arrows}
\usetikzlibrary{positioning}
\usetikzlibrary{patterns}
\usetikzlibrary{snakes}

\usepackage[active]{srcltx}

\usepackage{enumerate}
\numberwithin{equation}{section}

\title{The principal bundles over an inverse semigroup}

\author{Ganna Kudryavtseva}
\address{Ganna~Kudryavtseva, 
Faculty of Civil and Geodetic Engineering, University of Ljubljana, Jamova 2, 1000, Ljubljana, SLOVENIA; 
Institute of Mathematics, Physics and Mechanics,
Jadranska 19,
1000, Ljubljana,
SLOVENIA 
}
\email{ganna.kudryavtseva\symbol{64}fgg.uni-lj.si}
\email{ganna.kudryavtseva\symbol{64}imfm.si}
\author{Primo\v z \v Skraba}
\address{Primo\v z \v Skraba, 
Jo\v zef  Stefan Institute,
Jamova 39,
1000, Ljubljana,
SLOVENIA }
\email{primoz.skraba\symbol{64}ijs.si}

\thanks{The first author was partially partially funded by the EU project TOPOSYS (FP7-ICT-318493-STREP) and by ARRS grant P1-0288, and the second author was funded by the EU project TOPOSYS (FP7-ICT-318493-STREP)}

\begin{document}
\maketitle

\begin{abstract} This paper is a contribution to the development of the theory of representations of  inverse semigroups in toposes. It  continues the work initiated by Funk and Hofstra ~\cite{FH}. 
 For the  topos of sets, we show that torsion-free functors on Loganathan's category $L(S)$ of an inverse semigroup $S$ are equivalent to a special class of non-strict representations of $S$, which we call connected. We show that the latter representations form a proper coreflective subcategory of the category of all non-strict representations of $S$. We describe the correspondence between directed and pullback preserving functors on $L(S)$ and transitive and effective representations of $S$, as well as between filtered such functors and universal representations introduced by Lawson, Margolis and Steinberg. We propose a definition of a universal representation, or, equivalently, an $S$-torsor, of an inverse semigroup $S$ in the topos of sheaves ${\mathsf{Sh}}(X)$ on a topological space $X$.  We prove that the category of filtered functors from  $L(S)$ to the topos ${\mathsf{Sh}}(X)$ is equivalent to the category of universal representations of $S$ in ${\mathsf{Sh}}(X)$. We finally propose a definition of an inverse semigroup action 
 in an arbitrary Grothendieck topos, which arises from a functor on $L(S)$. \end{abstract}

\section{Introduction}

The classifying topos ${\mathcal{B}}(S)$ of an inverse semigroup $S$ has recently begun to be investigated \cite{F,FH,FLS,FS,KL,St}. This topos is by definition the presheaf topos over Loganathan's category $L(S)$ of $S$. There are several equivalent characterizations of this topos, cf. \cite{F,FS,KL}. An immediate question one can ask about ${\mathcal{B}}(S)$ is ``What does ${\mathcal{B}}(S)$ classify?''  A direct application of well-known results \cite[Theorems VII.7.2, VII.9.1]{MM} of topos theory, provides the answer: for an arbitrary Grothendieck topos ${\mathcal E}$, the presheaf topos ${\mathcal{B}}(S)$ classifies filtered functors $L(S)\to {\mathcal E}$.

\begin{theorem} The category of geometric morphisms from ${\mathcal E}$ to ${\mathcal{B}}(S)$ is equivalent to the category of filtered functors $L(S)\to {\mathcal E}$.
\end{theorem} 

The construction of the functors establishing the correspondence in the above theorem can be found in \cite{MM}. In particular, if $\gamma^*\colon {\mathcal{B}}(S)\to {\mathcal E}$ is the inverse image functor of a geometeric morphism, its composition with the Yoneda embedding $L(S)\to {\mathcal{B}}(S)$ is a filtered functor, and any such a functor is obtained this way. 

This answer, however, is not quite satisfactory. We expect more structure given that for  groups the category of filtered functors $G\to {\mathcal E}$ is equivalent to the category of $G$-{\em torsors}. The latter are just objects of ${\mathcal E}$ with a particular type of internal action of the group object obtained by applying the canonical constant sheaf functor $\Delta$ to $G$ \cite[VIII.2]{MM}. This naturally raises a question of how to define $S$-torsors, where $S$ is an inverse semigroup. The latter question was raised by Funk and Hofstra in \cite{FH}, where in \cite[Theorem 3.9]{FH} they show that, for the topos of sets,  $S$-torsors can be defined as (well-supported) transitive and free $S$-sets, where an inverse semigroup $S$-set is a homomorphism from $S$ to the symmetric inverse semigroup ${\mathcal{I}}(X)$. This description naturally generalizes  the description of $G$-torsors in the topos of sets. It is mentioned (without providing details) by Lawson, Margolis and Steinberg \cite{LMS} that $S$-torsors in the topos of sets are precisely universal representations of $S$ defined and systematically studied in~\cite{LMS}. 

Funk and Hofstra, in \cite[Definition 2.14]{FH},  proposed a definition of an $S$-torsor, where $S$ is an inverse semigroup, in an arbitrary Grothendieck topos ${\mathcal E}$. They also state an equivalence of categories between $S$-torsors in ${\mathcal E}$ and filtered functors $L(S)\to {\mathcal E}$ (\cite[Theorem 3.10]{FH}). Their approach is based on internalizing $S$ in ${\mathcal E}$ as a semigroup (rather than as an inverse semigroup). 
Implicitly,   \cite[Definition 2.14]{FH} considers actions  (in an arbitrary Grothendieck topos)  by partial bijections \cite{private_comm}. However, actions of by partial bijections in an arbitrary Grothendieck topos are not defined in \cite{FH}, nor, to the best or our knowledge, anywhere else in the literature. Therefore, 
while the main theorem \cite[Theorem 3.10]{FH} is correct, 
the ommision of this definition along with other details, can make some of the definitions (e.g., \cite[Definition 2.14]{FH}) and proofs in \cite{FH}  hard to follow or verify for the non-expert in topos theory. One of the main  goals in the present paper is to try to make the constructions as detailed, simple and explicit as possible and particularly tailored to researchers in semigroup theory. Additionally, we provide a counter-example to a claim in  \cite[Section 6]{FH} (see Section \ref{sec:fin} for details).

The  paper is structured as follows. Section~\ref{sec:prel} provides some preliminaries needed to read this paper (as well as suggestions of the literature for further reading).  In Section~\ref{sec:sets}, we focus on the topos of sets and describe (possibly non-strict) $S$-sets, where $S$ is an inverse semigroup, attached to  classes of functors from $L(S)$ to this topos. Some of these results were already given in \cite{FH}, but we provide detailed proofs. We introduce a class of  {\em connected} non-strict $S$-sets and prove that they are in a categorical equivalence with torsion-free functors on  $L(S)$ (Theorem \ref{th:equiv}). We then show that connected non-strict $S$-sets form a proper coreflective subcategory of the category of all non-strict $S$-sets which corrects \cite[Proposition 3.6]{FH} (see Example \ref{ex:ce} and Proposition \ref{prop:ce}).  We also discuss the connection of transitive and universal $S$-sets with appropriate classes of functors on $L(S)$. This in particular leads to a new perspective on the classical result due to Schein \cite{Sch} on transitive and effective representations of an inverse semigroup. In Section~\ref{sec:bundles}, we define $S$-torsors in the topos of sheaves ${\mathsf{Sh}}(X)$ over a topological space $X$ and prove that these are categorically equivalent to filtered functors $L(S)\to {\mathsf{Sh}}(X)$ (Theorem \ref{th:sheaves}).  It follows that in the topos ${\mathsf{Sh}}(X)$ the classifying topos ${\mathcal B}(S)$ classifies universal $S$-bundles.This can be seen as an instance of \cite[Theorem 3.10]{FH}, with more details which hopefully provide a better insight into why this works. 

Finally, in Section~\ref{sec:fin}, we outline an approach, which is substantially different from that used in \cite{FH}, to the notion of an $S$-set in an arbitrary Grothendieck  topos. We start from a functor on $L(S)$, construct an objects of action as a certain colimit (similarly as this is done for the topos of set) and then lift $S$ to the topos ${\mathcal H}$-class-wise, that is, we consider objects $\Delta H$, where $H$ is an ${\mathcal H}$-class of $S$ and $\Delta$ is the constant sheaf functor. It would be interesting to connect and compare this approach with the approach proposed in \cite{FH}.

An important task which remains for future investigation is to further develop the general theory of actions of inverse semigroups by partial bijections in arbitrary Grothendieck toposes extending \cite{FH} and the present paper to the level of corresponding well-established theory of group actions \cite[V.6, VIII.2]{MM}. 

\section{Preliminaries}\label{sec:prel}

For more complete exposition on inverse semigroups, we refer the reader to \cite{L}, on categories to \cite{Awo, MacL}, and on toposes to \cite{MM,M}. 

\subsection{Inverse semigroups and their representations} \label{subs:inv} Let $S$ be an inverse semigroup. By $E(S)$, we denote the semilattice of idempotents of $S$.  For $s\in S$ we write ${\mathbf{d}}(s)=s^{-1}s$ and ${\mathbf{r}}(s)=ss^{-1}$. These idempotents are abstractions of the notions of the domain and the range idempotents, respectively, of a partial bijection. The natural partial order on $S$ is defined by  $s\leq t$ if and only if $s=te$ for some $e\in E(S)$. For  $X\subseteq S$, we write
$$
X^{\uparrow}=\{s\in S\colon s\geq x \text{ for some } x\in X\}. 
$$
The set $X^{\uparrow}$ is sometimes called the {\em (upward) closure} of $X$. The set $X$ is {\em closed} if $X^{\uparrow}=X$.

For a set $X$, let ${\mathcal I}(X)$ denote the {\em symmetric inverse semigroup} on $X$ which consists of all bijections between subsets of $X$ (we refer to such maps  as {\em partial bijections}). 
If $s\in {\mathcal I}(X)$ we set ${\mathrm{dom}}(s)$ and ${\mathrm{im}}(s)$ to be the domain and the image of $s$. 

A {\em representation} of an inverse semigroup $S$ on a set $X$, is an inverse semigroup homomorphism $\theta\colon S\to {\mathcal I}(X)$. 
Given a such a representation, we have a left action of $S$ on $X$ by partial bijections such that $s\cdot x$ is defined if and only if $x\in {\mathrm{dom}}(\theta(s))$ in which case $s\cdot x= \theta(s)(x)$.
We say that $(X,\theta)$ is a {\em left} $S$-{\em set}.  Where $\theta$ is clear, we will write $(X,\theta)$ as simply $X$. Unless otherwise stated, we assume that actions are left actions, and we refer to left $S$-sets as $S$-{\em sets}. Throughout the paper, we assume that the $S$-sets are {\em effective}, meaning that for every $x\in X$ there exists some $s\in S$ such that $s\cdot x$ is defined. 

An $S$-set $(X,\mu)$ is called {\em transitive} if for any $x,y\in X$ there is $s\in S$ such that $\mu(s)(x)=y$. It is called 
{\em free}, if the equality $\mu(s)(x)=\mu(t)(x)$ implies that there is $c\leq s,t$ such that $\mu(c)(x)=\mu(s)(x)$. 
Finally, we call a transitive and free $S$-set an $S$-{\em torsor}.

\subsection{Toposes in a nutshell} By a {\em topos}, we restrict ourselves to a {\em Grothendieck topos}, that is, a category ${\mathcal{E}}$ that satisfies the {\em Giraud's axioms}. We refer the reader, for example, to \cite[1.1]{M} for a detailed introduction to the notion of a topos. For our purposes, we do not need to recount the definition of a topos. It is important however to mention the following examples of toposes:
\begin{enumerate}
\item The category ${\mathsf{Sets}}$ of sets and maps between sets.
\item The category ${\mathrm{Et}}(X)$ of \'etale spaces over a topological space $X$.
\item The category ${\mathcal B}({\mathcal{C}})$ of presheaves of sets $F\colon {\mathcal{C}}^{op} \to {\mathsf{Sets}}$ over a small category ${\mathcal{C}}$.
\end{enumerate}

Let us look at these examples at greater detail. An \'etale space over a topological space $X$ 
is a triple $(E,p,X)$ where $E$ is a topological space and $p\colon E\to X$ is a local homeomorphism. A {\em morphism}
$ (E,p,X) \to (G,q,X)$ between \'etale spaces 
is a continuous map $\alpha\colon E\to G$ such that $q\alpha=p$. Given the well known equivalence between \'etale spaces and sheaves, the topos ${\mathrm{Et}}(X)$ is equivalent to the topos ${\mathsf{Sh}}(X)$ of sheaves over $X$. From the topos of sheaves ${\mathsf{Sh}}(X)$ one can recover the frame of opens of  $X$, and thus, if $X$ is a sober space, $X$ itself can  also be recovered \cite{M, MM}. It follows that a topos can be thought of as a generalization of a (sober) topological space. Bearing this in mind, it is useful (for example, to interpret the definition of a point of a topos) to consider the topos ${\mathsf{Sets}}$ as an analogue of a one-point space. 

Turning to the third example, a {\em presheaf of sets} over a small category ${\mathcal{C}}$ is a contravariant functor $F$ from ${\mathcal{C}}$ to the category of sets ${\mathsf{Sets}}$, $F\colon {\mathcal{C}}^{op} \to {\mathsf{Sets}}$. If $\alpha\colon c\to d$ is a morphism in ${\mathcal{C}}$, then the map
$F(\alpha)\colon F(d)\to F(c)$ is called the {\em translation map} along $\alpha$.
Let $F,G\colon {\mathcal{C}}^{op} \to {\mathsf{Sets}}$ be presheaves of sets. By a morphism from $F$ to $G$, we mean a natural transformation $\pi$ from $F$ to $G$, that is, a collection of maps, $\pi_c\colon F(c)\to G(c)$, where $c$ runs through the objects of ${\mathcal{C}}$, which commute with the translation maps. The topos ${\mathcal B}({\mathcal{C}})$ is called {\em the classifying topos} of the small category ${\mathcal{C}}$. 

For a detailed verification that each of our examples satisfies the Giraud's axioms, we refer the reader to \cite{M}.

\subsection{The category of elements of a functor}

Let ${\mathcal C}$ be a small category and $P\colon {\mathcal C}\to {\mathsf{Sets}}$ a covariant functor. The {\em category of elements of} $P$ is the category $\int_{\mathcal C}P$ whose objects are all pairs $(C,p)$ where $C$ is an object of ${\mathcal C}$ and $p\in P(C)$. Its morphisms $(C,p)\to (C',p')$ are those morphisms $u\colon C\to C'$ of ${\mathcal C}$ for which $P(u)(p)=p'$. The category of elements $\int_{\mathcal C} P$ of a contravariant functor $P\colon {\mathcal C}^{op}\to {\mathsf{Sets}}$ is defined similarly.

\subsection{Filtered  and directed categories and functors}\label{subs:2.5}

A small category $I$ is called {\em filtered} if it satisfies the following axioms:
\begin{enumerate}
\item[(F1)] $I$ has at least one object.
\item[(F2)] For any two objects $i,j$ of $I$ there is a diagram $i\leftarrow k \to j$ in $I$, for some object $k$.
\item[(F3)] For any two parallel arrows $i\rightrightarrows j$ there exists a commutative diagram $k\to i\rightrightarrows j$ in $I$.
\end{enumerate}

Equivalently, a small category $I$ is filtered if for any finite diagram in $I$ there is a cone on that diagram.
A small category $I$ is called {\em directed} if it satisfies axioms (F1) and (F2) above.

A covariant functor $A\colon {\mathcal C}\to {\mathsf{Sets}}$ is called a {\em filtered functor} (resp. a {\em directed functor}) if its category of elements $\int_{\mathcal C} A$ is a filtered category (resp. a directed category). 

\subsection{Geometric morphisms} Let ${\mathcal E}, {\mathcal F}$ be toposes. A {\em geometric morphism} $f\colon {\mathcal F} \to {\mathcal E}$ consists of a pair of functors
$$
f^*\colon {\mathcal E} \to {\mathcal F} \text{ and } f_*\colon {\mathcal F} \to  {\mathcal E},
$$
called the {\em inverse image functor} and the {\em direct image functor}, respectively, such that the following two axioms are satisfied:
\begin{enumerate}
\item[(GM1)] $f^*$ is a left adjoint to $f_*$.
\item[(GM2)] $f^*$ is left exact, that is, it commutes with finite limits.
\end{enumerate}

Since $f^*$ is a left adjoint, it commutes with colimits (by the dual to the well-known RAPL theorem \cite{Awo}).
It follows from the uniqueness of adjoints that a geometric morphism $f\colon {\mathcal F} \to {\mathcal E}$ is determined by its inverse image functor $f^*\colon {\mathcal E}\to {\mathcal F}$ which is required to commute with any colimits and finite limits.

Let $X,Y$ be topological spaces and $f\colon X\to Y$ a continuous map. This  gives rise to a functor $f^*\colon {\mathrm{Et}}(Y)\to {\mathrm{Et}}(X)$, as follows.  Let $(E,p,Y)$ be an \'etale space over $Y$ and put 
$$X\times_Y E =\{(x,e)\in X\times E\colon f(x)=p(e)\}.$$
Then the projection to the first coordinate $\pi_1\colon X\times_Y E \to X$ is a local homeomorphism. Indeed, assume that $(x,e)\in X\times_Y E$ and let $A$ be a neighborhood of $e$ such that $A$ is homeomorphic to $p(A)$. 
Then the set
$$
\{(x,t)\in X\times_Y E\colon t\in A\}
$$
is homeomorphic to $f^{-1}(p(A))$ via $\pi_1$. 
The local homeomorphism $\pi_1$ is said to be obtained by {\em pulling} $p$  {\em back along} $f$. We set 
$$f^*(E,p,Y)=(X\times_Y E, \pi_1, X).$$
It is easy to see that $f^*$ preserves colimits and finite limits and thus gives rise to a geometric morphism $(f^*,f_*)$ from ${\mathrm{Et}}(X)$ to ${\mathrm{Et}}(Y)$. For sober spaces $X$ and $Y$ this construction gives rise to a bijective correspondence between continuous maps from $X$ to $Y$ and geometric morphisms from  ${\mathrm{Et}}(X)$ to ${\mathrm{Et}}(Y)$. Thus, as toposes can be looked at as generalizations of topological spaces, geometric morphisms between toposes are generalizations of continuous maps.

\subsection{The constant sheaf functor and the global section functor}

For any topos ${\mathcal E}$, there is a unique (up to isomorphism) geometric morphism $\gamma\colon {\mathcal E}\to {\mathsf{Sets}}$, given by
$$
\gamma^*(S)=\sum_{s\in S}1, \,\, \gamma_*(E)={\mathrm{Hom}}_{\mathcal E}(1,E),
$$
where $1$ denotes the terminal object of ${\mathcal E}$. The inverse image part $\gamma^*$ of $\gamma$ is usually denoted by $\Delta$ and is called the {\em constant sheaf functor}, and the direct image part $\gamma_*$ is usually denoted by $\Gamma$ and is called the {\em global section functor.} This geometric morphism may be looked at as an analogue of the only continuous map from a  topological space $X$ to a one-element topological space.

\subsection{Filtered functors and geometric morphisms}

A {\em point} of a topos ${\mathcal E}$ is a geometric morphism $\gamma: {\mathsf{Sets}}\to {\mathcal E}$. This is parallel to looking at a point of a topological space $X$ as an inclusion of a one-element space into $X$. Note that such an inclusion $i\colon \{x\}\to X$ defines a filter $F$ in $X$ consisting of those $A\in X$ such that $i(x)\in A$. We now describe how this idea can be extended to a correspondence between points of the classifying topos of a category and filtered functors on this category.

 Let ${\mathcal C}$ be a small category, and let $A\colon {\mathcal C}\to {\mathsf{Sets}}$ be a functor.  We describe a  construction to be found in \cite{MM} of a pair of adjoint functors
$f^*\colon {\mathcal{B}}({\mathcal C})\to {\mathsf{Sets}}$ and $f_*\colon {\mathsf{Sets}}\to {\mathcal{B}}({\mathcal C})$. The functor $f_*$ is easier to define and  thus we start from its description.  We have $f_*=\underline{\mathrm{Hom}}_{\mathcal C}(A,-)$, where the latter is the presheaf defined for each set  $R$ and $C\in {\mathcal C}$ by
$$
\underline{\mathrm{Hom}}_{\mathcal C}(A,R)(C)={\mathrm{Hom}}_{\mathsf{Sets}}(A(C),R).
$$

 For a presheaf $P\in {\mathcal{B}}({\mathcal C})$, we define $f^*(P)$ to be the colimit
$$
f^*(P)=\lim_{\longrightarrow}\left(\int_{\mathcal C}P\stackrel{\pi_1}{\to} {\mathcal C} \stackrel{A}{\to} {\mathsf{Sets}}\right),
$$
where $\pi_1(C,p)=C$. This colimit is the set which we denote by $P\otimes_{\mathcal{C}}A$. It is the quotient of the set
$\bigcup_{C\in {\mathcal{C}}}(P(C)\times A(C))$ by the equivalence relation $\sim$ generated by
$$
(pu,a')\sim (p,ua'), \, p\in P(C), u\colon C\to C', a'\in A(C'),
$$
where we denote $pu=P(p)(u)$ and $ua'=A(u)(a')$. We denote the elements of $P\otimes_{\mathcal{C}}A$ by $p\otimes a$ and treat them as tensors where ${\mathcal{C}}$ `acts' on $P$ on the right and on $A$ on the left.

The described adjoint pair $(f^*, f_*)$ is not in general a geometric morphism between toposes. By definition, it is a geometric morphism if and only if the tensor product functor $f^*$ is left exact. If this condition holds, the functor $A$ is called {\em flat}. Flat functors can be characterized precisely as filtered functors \cite[Theorem VII.6.3]{MM}.

Let ${\mathsf{Filt}}({\mathcal C})$ denote the category of filtered functors ${\mathcal C}\to {\mathsf{Sets}}$, where morphisms are natural transformations, and ${\mathsf{Geom}}({\mathsf{Sets}}, {\mathcal{B}}({\mathcal C}))$ the category of geometric morphisms from ${\mathsf{Sets}}$ to the classifying topos ${\mathcal{B}}({\mathcal C})$ of ${\mathcal C}$ (or, equivalently the points of ${\mathcal{B}}({\mathcal C})$), where morphisms are natural transformations between the inverse image functors. 

\begin{theorem}[{\cite[Theorem VII.5.2]{MM}}] \label{th:filt}  There is an equivalence of categories
$${\mathsf{Filt}}({\mathcal C})\, \,\drightleftarrow{\tau}{\rho}\, \,{\mathsf{Geom}}({\mathsf{Sets}}, {\mathcal{B}}({\mathcal C})) 
$$
where the functors $\tau$ and $\rho$ are defined, for a filtered functor $A\colon {\mathcal C}\to {\mathsf{Sets}}$ and a point $f\in {\mathsf{Geom}}({\mathsf{Sets}}, {\mathcal{B}}({\mathcal C})) $, by
$$
\tau(A)^*=- \otimes_{\mathcal C} A, \,\, \tau(A)_*= \underline{\mathrm{Hom}}_{\mathcal C}(A,-),
$$
$$
\rho(f)=f^*\cdot {\mathbf{y}}\colon {\mathcal C}\to {\mathcal{B}}({\mathcal C})\to {\mathsf{Sets}},
$$
where ${\mathbf{y}}$ denotes the Yoneda embedding of ${\mathcal C}$ into ${\mathcal{B}}({\mathcal C})$.
\end{theorem}

We remark that  Theorem \ref{th:filt} remains valid in a wider setting where the topos ${\mathsf{Sets}}$ is replaced by an arbitrary topos ${\mathcal E}$. To formulate this result, known as Diaconescu's theorem, one needs a suitable definition of a filtered functor from a small category to a topos, such that being filtered is equivalent to being flat, cf. \cite[VII.8]{MM}. For our purposes, we will need filtered functors to the topos of sheaves over a topological space, which we discuss in Section \ref{sec:bundles}.

\subsection{Principal group bundles and group torsors}
The connection between filtered functors on a small category and geometric morphisms to the classifying topos is a well known and fundamental result in topos theory. In the special case where the category is a group, denote it by $G$, it is known \cite[VIII.2]{MM} that the category of filtered functors $G\to {\mathcal E}$, where ${\mathcal E}$ is an arbitrary topos, is equivalent to the category of so-called $G$-torsors over ${\mathcal E}$. A $G$-{\em torsor over}  ${\mathcal E}$ is an object $T$ of ${\mathcal E}$ equipped with an internal action (cf. \cite[V.6]{MM}) of a group object $\Delta G$ on it which satisfies some technical conditions. We remind the reader that $\Delta G$ is the value of the constant sheaf functor $\Delta\colon {\mathsf{Sets}}\to {\mathcal E}$ on $G$. Since $G$ has a structure of  a group,  $\Delta G$ inherits a structure of an internal group in ${\mathcal E}$. The aforementioned technical conditions arise as an abstraction of the well-known notion of a $G$-torsor in the topos of sheaves over a topological space $X$. In this setting, a $G$-torsor is just synonymous with a principal $G$-bundle.  A {\em principal} $G$-{\em bundle} over a topological space $X$ can be characterized as an \'etale space $(E,p,X)$ with a continuous action $G\times E\to E$ over $X$ such that
\begin{enumerate}
\item[(i)] For each point $x\in  X$ the stalk $E_x=p^{-1}(x)$ is non-empty;
\item[(ii)] Stalks are invariant under the action;
\item[(iii)] The action map $G\times E_x \to E_x$ on each stalk $E_x$ is free meaning that $g\cdot x=x$ implies that $g$ is the identity element;
\item[(iv)] The action map $G\times E_x \to E_x$ on each stalk $E_x$ is free transitive, meaning that for every $a,b\in E_x$ there exists some $g\in G$ such that $g\cdot a=b$. 
\end{enumerate}
A $G$-torsor in the topos of sets is a set $X$ equipped with a free and transitive action of $G$ on it. Such a set is in a bijection with $G$ and the action is equivalent to the action of $G$ on itself by left translations (this is a direct consequence of the elementary fact that a transitive group action is equivalent to the left action on the set of cosets over the stabilizer of any point). 
 In particular, up to isomorphism, there is only one $G$-torsor in the topos of sets.

The equivalence between filtered functors $G\to {\mathcal E}$ and $G$-torsors over ${\mathcal{E}}$, together with Theorem \ref{th:filt},  yields the result that the classifying topos ${\mathcal{B}}(G)$ classifies $G$-torsors in the sense that for an arbitrary topos ${\mathcal E}$, there is a categorical equivalence between geometric morphisms ${\mathcal E}\to {\mathcal{B}}(G)$ and $G$-torsors over ${\mathcal{E}}$. This parallels the topological result that the classifying space of $G$ classifies principal $G$-bundles.

\section{Covariant functors $L(S)\to {\mathsf{Sets}}$ vs representations of $S$ in ${\mathsf{Sets}}$}\label{sec:sets}
 
The relationship between various classes of (possibly non-strict) representations of an inverse semigroup $S$ in the topos of sets and covariant functors $L(S)\to {\mathsf{Sets}}$  was first observed and studied by Funk and Hofstra in \cite{FH}. 
In particular, they observe that filtered functors on $L(S)$ correspond to representations of $S$ which are transitive and free (\cite[Theorem 3.9]{FH}, though these representations are wrongly referred to as another kind of representations). They also consider torsion-free and pullback preserving functors. In this section we prove that torsion-free functors on $L(S)$ correspond to a class of $S$-sets which we call {\em connected}. (This corrects an inaccuracy in  \cite[Proposition 3.6]{FH}.)  
We also put in correspondence directed functors on $L(S)$ and transitive effective representations of $S$, providing a different approach to the classical theory due to Schein \cite{Sch}. Finally, we explain that filtered functors on $L(S)$ correspond to a class of representations, called {\em universal}, which were introduced and studied by Lawson, Margolis and Steinberg in \cite{LMS}.

\subsection{Torsion-free functors and connected non-strict representations} 

A map $\varphi\colon S\to T$ between inverse semigroups is called a {\em prehomomorphism} if $\varphi(ab)\leq \varphi(a)\varphi(b)$ for any $a,b\in S$. A prehomomorphism $S\to {\mathcal{I}}(X)$ will be called a {\em non-strict} representation of $S$. Similarly as representations correspond to $S$-sets, non-strict representations correspond to {\em non-strict} $S$-{\em sets}\footnote{Note that in \cite{FH} non-strict $S$-sets are referred to as $S$-sets, and  $S$-sets are referred to as strict $S$-sets.}, where the latter means a set $X$ together with a partial map $S\times X \to X$, $(s,x)\mapsto s\cdot x$, where defined, such that if $st\cdot x$ is defined then $t\cdot x$ and $s\cdot (t\cdot x)$ are defined and $st\cdot x= s\cdot (t\cdot x)$. Just as $S$-sets, the non-strict $S$-sets we consider are effective.

 The following constructions connecting non-strict $S$-sets and some covariant functors on $L(S)$ were introduced in \cite{FH}. We give here their slightly different but equivalent description. We also provide more details and notice the  property of connectedness. 
 
 Let $(X,\mu)$ be a non-strict $S$-set where $(s,x)\mapsto \mu(s,x)=s\cdot x$, where defined. For each $e\in E(S)$ let $\Phi(X,\mu)(e)$ be the domain of the action of $e$, that is to say,
$$
\Phi(X,\mu)(e)=\{x\in X\colon e\cdot x\text{ is defined}\}.
$$

If $(f,s)$ is an arrow in $L(S)$, we define $\Phi(X,\mu)(f,s)$ to be the map from $\Phi(X,\mu)({\mathbf{d}}(s))$ to $\Phi(X,\mu)(f)$ given by $x\mapsto s\cdot x$. Since $e\cdot x$ is defined and $e={\mathbf{d}}(s)=s^{-1}s$, we have that $s\cdot x$ is defined. 
Observe that $s \cdot x=(fs)\cdot x$, so that $f\cdot (s\cdot x)$ is defined. Thus $s\cdot x\in \Phi(X,\mu)(f)$.  We have constructed the covariant functor $\Phi(X,\mu)$ on $L(S)$. We need to record that the functor $\Phi(X,\mu)$ has one important property. We first define this property. 

 Assume that $F\colon L(S)\to {\mathsf{Sets}}$ is a functor and put $\Psi(F)$ to be the colimit of the following composition of functors:
$$
E(S)\longrightarrow L(S)\stackrel{F}{\longrightarrow} {\mathsf{Sets}}.$$

 This colimit is, by definition, equal to the quotient set
\begin{equation}\label{eq:colimit} \Psi(F)= \left(\bigcup_{e\in E(S)}\{e\}\times F(e)\right)/\sim,
\end{equation}
where the equivalence $\sim$ on $\bigcup_{e\in E(S)}\{e\}\times F(e)$ is generated by $(e,x)\sim (e',F(e',e)(x))$.  
The functor $F$ is called {\em torsion-free} if $(e,x)\sim (e,y)$ implies that $x=y$.

\begin{lemma} The constructed functor $\Phi(X,\mu)\colon L(S)\to {\mathsf{Sets}}$ is torsion-free.
\end{lemma}

\begin{proof} This follows from the definition of $\sim$ since 
$$\Phi(X,\mu)(e',e)(x)=e'\cdot x=e\cdot x=x$$ for any $e'\geq e$ in $E$ and any $x\in  X$ such that $e\cdot x$ is defined.
\end{proof}

We have therefore assigned to $(X,\mu)$ a torsion-free functor $\Phi(X,\mu)\colon L(S)\to {\mathsf{Sets}}$.
We now describe the reverse direction. Assume that $F$ is a torsion-free functor $L(S)\to {\mathsf{Sets}}$. By $[e,x]$ we will denote the $\sim$-class of $(e,x)$. 
For $s\in S$ and $\alpha\in \Psi(F)$ we define
\begin{equation}\label{eq:action} s\circ\alpha=\left\lbrace\begin{array}{ll}[{\mathbf{r}}(s), F({\mathbf{r}}(s),s)(x)],& \text{ if } \alpha=[{\mathbf{d}}(s),x];\\ \text{undefined,} & \text{otherwise.} \end{array}\right.
\end{equation}

If $\alpha\in \Psi(F)$ we define

\begin{equation}\label{eq:connected}\pi_1(\alpha)=\{e\in E\colon \text{ there is some }(e,x)\in \alpha\}=\{e\in E\colon e\circ \alpha \text{ is defined}\}.
\end{equation}
It follows that $s\circ \alpha$ is defined if and only if ${\mathbf{d}}(s)\in\pi_1(\alpha)$.

\begin{lemma}\label{lem:lem3}\mbox{}
\begin{enumerate}
\item The map $\alpha\mapsto s\circ\alpha$, given by \eqref{eq:action}, is injective on its domain.
\item The assignment \eqref{eq:action} defines on $\Psi(F)$ the structure of a non-strict $S$-set $(\Psi(F), \nu)$.
\item  For any $\alpha\in \Psi(F)$ and $e,f\in \pi_1(\alpha)$,  there are $$e=e_1,e_2,\dots, e_k=f$$ in $\pi_1(\alpha)$ such that $e_i\geq e_{i+1}$ or $e_i\leq e_{i+1}$ for all admissible $i$. 
\end{enumerate} 
\end{lemma}

\begin{proof} (1) Follows from  \eqref{eq:action}, since $F$ is torsion-free and thus all the translation maps are injective.  

(2) We use the fact that a map $\varphi\colon S\to T$ between inverse semigroups is a prehomomorphism if and only if $\varphi(st)=\varphi(s)\varphi(t)$ for any $s,t$ such that ${\mathbf{r}}(t)={\mathbf{d}}(s)$ and $\varphi(ef)\leq \varphi(e)\varphi(f)$ for any $e,f\in E(S)$. It is immediate from \eqref{eq:action} that both of these conditions hold for  $\Psi(F)$.

(3) Follows from the construction of $\Psi(F)$ and \eqref{eq:action}.

\end{proof}

Since the set $\pi_1(\alpha)$ is expressable in terms of the action, as is given in \eqref{eq:connected}, we can define a non-strict $S$-set $X$, $(s,x)\mapsto s\cdot x$, where defined, to be {\em connected} if for any $x\in X$ and any $e,f\in E$ such that $e\cdot x$ and $f\cdot x$ are defined, there is a sequence of idempotents $e=e_1,e_2,\dots, e_k=f$, called a {\em connecting sequence over} $x$, such that $e_i\cdot x$ is defined and  $e_i\geq e_{i+1}$ or $e_i\leq e_{i+1}$ for all admissible $i$.

\begin{example} {\em If $X$ is an $S$-set (that is, given by a homomomorphism),  it is connected with  $e,ef,f$ being a connecting sequence between $e$ and $f$ over any $x$ such that $e\cdot x$ and $f\cdot x$ are defined. }
\end{example}

\begin{example} {\em If $S$ is a monoid,  any non-strict $S$-set is connected with $e,1,f$ being a connecting sequence between $e$ and $f$, again over any $x$ such that $e\cdot x$ and $f\cdot x$ are defined.}
\end{example}

It is not true that every non-strict $S$-set is connected, as the following  example shows.

\begin{example}\label{ex:ce} {\em Let $S=\{e,f,g\}$ be a three-element semilattice, given by the following Hasse diagram:
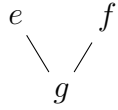
\begin{center}
\begin{tikzpicture}[scale=2]
\node (e) {$e$};
\node(aux) [node distance=0.6cm, right of=e] {};
\node (f) [node distance=1.2cm, right of=e] {$f$};
\node (g) [node distance=1cm, below of=aux] {$g$};
\path[-]
(e) edge node[above]{} (g)
(f) edge node[above]{} (g);
\end{tikzpicture}
\end{center}

Let $X=\{1,2\}$ and define the domains of action of $e$ and $f$ to be equal $\{1,2\}$, and the domain of action of $g$ to be equal $\{1\}$ (that is, $e$ and $f$ act by the identity map on $\{1,2\}$, and $g$ by the identity map on $\{1\}$). Thus $X$ becomes a non-strict $S$-set. It is however not connected, as both $e\cdot 2$ and $f\cdot 2$ are defined but there is no connecting sequence between $e$ and $f$ over $2$ as $g\cdot 2$ is undefined.}
\end{example}

In view of Lemma \ref{lem:lem3}, it follows that the non-strict $S$-set from Example \ref{ex:ce} can not be equal $\Psi(F)$ for any torsion-free functor $F$ on $L(S)$.  

We now describe the correspondence between morphisms of non-strict  $S$-sets and natural transformations of torsion-free functors on $L(S)$.  Assume we are given non strict $S$-sets $(X,\mu)$, $(s,x)\mapsto s\cdot x$, where defined, and $(Y,\nu)$, $(s,x)\mapsto s\circ x$, where defined.  A {\em morphism} from $(X,\mu)$ to $(Y,\nu)$ is  
a map $f\colon X\to Y$ such that if $s\cdot x$ is defined then $s\circ  f(x)$ is also defined and
$$
f(s\cdot x) = s\circ (f(x)).
$$

Let $f\colon (X,\mu)\to (Y,\nu)$ be a morphism, $e\in E$ and $x\in \Phi(X, \mu)(e)$. Then $f(x)\in \Phi(Y, \nu)(e)$ which defines a map 
$$\widetilde{f}_e\colon \Phi(X, \mu)(e)\to \Phi(Y, \nu)(e).$$
It is immediate that the maps $\widetilde{f}_e$ commute with the translation maps along any $(f,s)\in L(S)$ and thus define a natural transformation $\widetilde{f}\colon \Phi(X, \mu)\to \Phi(Y, \nu)$. We set $\Phi(f)=\widetilde{f}$.

In the reverse direction, let $F$ and $F'$ be torsion-free functors $L(S)\to {\mathsf{Sets}}$ and let
$\alpha\colon F\to F'$ be a natural transformation. Let $\alpha(e)$ denote the component of $\alpha$ at $e$. Further, let $\sim$ denote the congruence on the set $\bigcup_{e\in E(S)}\{e\}\times F(e)$ which defines the set $\Psi(F)$, and $\sim'$ denote a similar congruence which defines the set $\Psi(F')$.
\begin{lemma}
Let $x\in F(e)$, $y\in F(f)$ and $(e,x)\sim (f,y)$. Then $(e,\alpha(e)(x))\sim' (f,\alpha(f)(y))$.
\end{lemma}
\begin{proof}
Without loss of generality, we may assume that $e\leq f$ and that $y=F(f,e)(x)$. Since compotents of $\alpha$ commute with the translation maps, we can write
$$
\alpha(f)(y)=\alpha(f)(F(f,e)(x))=F'(f,e)(\alpha(e)(x)),
$$
which yields that $(e,\alpha(e)(x))\sim' (f,\alpha(f)(y))$. 
\end{proof}

The proved lemma shows that the assignment $[e,x]\mapsto [e,\alpha(e)(x)]$ results in a well-defined map 
$$
\widetilde{\alpha}\colon \Psi(F)\to \Psi(F').
$$

\begin{lemma}
The map $\widetilde{\alpha}$ is a morphism of non-strict  $S$-sets.
\end{lemma}

\begin{proof} 
Let the structure of  an $S$-set on $\Psi(F)$ (defined in \eqref{eq:action}) be given by $(s,\alpha)\mapsto s\circ \alpha$, where defined, and that on $\Psi(F')$ be given by $(s,\alpha)\mapsto s*\alpha$, where defined.
Let $s\in S$ and assume that $s\circ [e,x]$ is defined. We may then assume that $e={\mathbf{d}}(s)$.
Then $\widetilde{\alpha}([{\mathbf{d}}(s),x])=[{\mathbf{d}}(s),\alpha({\mathbf{d}}(s))(x)]$ showing that $s*\widetilde{\alpha}([{\mathbf{d}}(s),x])$ is defined, too. The proof is completed by the following calculation using the fact that components of $\alpha$ commute with the translation maps:
$$
\widetilde{\alpha}(s\circ ([{\mathbf{d}}(s),x]))=\widetilde{\alpha}([{\mathbf{r}}(s),F({\mathbf{r}}(s),s)(x)])=
[{\mathbf{r}}(s), \alpha({\mathbf{r}}(s))F({\mathbf{r}}(s),s)(x)];
$$
$$
s*\widetilde{\alpha}([{\mathbf{d}}(s),x])=s*[{\mathbf{d}}(s), \alpha({\mathbf{r}}(s))(x)]=[{\mathbf{r}}(s),F'({\mathbf{r}}(s),s)(\alpha({\mathbf{d}}(s),x))].
$$
\end{proof}

We set $\Psi(\alpha)=\widetilde{\alpha}$.
Let ${\mathsf{Repr}}(S)$ denote the category of all non-strict  $S$-sets, ${\mathsf{ConRepr}}(S)$ the category of all connected non-strict $S$-sets and ${\mathsf{TF}}(L(S))$  the category of torsion-free functors on $L(S)$.

It is routine to verify that the assignments $\Phi\colon {\mathsf{Repr}}(S)\to {\mathsf{TF}}(L(S))$ and $\Psi\colon {\mathsf{TF}}(L(S))\to {\mathsf{ConRepr}}(S)$ are functorial.  We denote the restriction of the functor $\Phi$ to the category ${\mathsf{ConRepr}}(S)$ by $\Phi'$. 
We obtain the following result.

\begin{theorem}\label{th:equiv} There is an equivalence of categories
$${\mathsf{ConRepr}}(S)  \,\,
\drightleftarrow{\Phi'}{\Psi} \,\, {\mathsf{TF}}(L(S)).
$$
\end{theorem}
\begin{proof} 
Let $F\colon L(S)\to {\mathsf{Sets}}$ be a torsion-free functor and show that $F$ is naturally isomorphic to the functor
$\Phi\Psi(F)$. By construction, for $e\in E$ we have 
$$\Phi\Psi(F)(e)=\{[e,x]\colon e\cdot x \text{ is defined}\}.$$
Clearly, the maps $\tau_e\colon x\to [e,x]$, $x\in F(e)$, $e\in E$, are bijections. In addition, 
these maps commute with the translation maps because for any arrow $(f,s)$ in $L(S)$, we have
$$
[{\mathbf{d}}(s),x]\stackrel {\Phi\Psi(F)(f,s)}{\xrightarrow{\hspace*{1.2cm}}} [f, F(f,s)(x)].
$$
It follows that we have constructed a natural isomorphism $\tau\colon F\to \Phi\Psi(F)$.

In the reverse direction, let $(X,\mu)$ be a connected non-strict $S$-set, where $(s,x)\mapsto s\cdot x$, where defined. The elements of the set $\Psi\Phi(X,\mu)$ are equivalence classes $[e,x]$ where $e\in E$ and $e\cdot x$ is defined. We define the map 
\begin{equation}\label{eq:beta}
\beta_{\mu}\colon \Psi\Phi(X,\mu)\to X
\end{equation}
 by
$[e,x]\mapsto x$. Let $x\in X$. Since $\mu$ is effective, there exists an $s\in S$ such that $s\cdot x$ is defined. But then ${\mathbf d}(s)\cdot x$ is defined, as well, which implies that the map $\beta_{\mu}$ is surjective. To show injectivity of $\beta_{\mu}$, we note that the condition $[e,x] \neq [f,x]$ is equivalent to the claim that there is no connecting sequence between $e$ and $f$, which does not happen as $\mu$ is connected. We obtain that the non-strict $S$-set $\Psi\Phi(X,\mu)$, $[e,x]\mapsto s\circ [e,x]$, where defined, is equivalent to $(X,\mu)$. Indeed, $s\circ [e,x]$ is defined if and only if $s\cdot x$ is defined, and in the case where $s\cdot x$ is defined we have the equality $$s\circ [e,x]=[{\mathbf{r}}(s), s\cdot x].$$ Moreover, this equivalence is natural in $(X,\mu)$.
\end{proof}

 \begin{corollary}[\cite{FH}] \label{cor:strict} The category of  $S$-sets is equivalent to the category of pull-back preserving functors on $L(S)$.
 \end{corollary}
 
\begin{proof} The statement follows from Theorem \ref{th:equiv} using the facts that any $S$-set is connected, and that a non-strict $S$-set $(X,\mu)$ is strict if and only if $\mu(ef)=\mu(e)\mu(f)$ for any $e,f\in E(S)$.
\end{proof}

We now establish a relationship between all non-strict $S$-sets and those of them which are connected. Recall that a subcategory ${\mathcal A}$ of a category ${\mathcal B}$ is called {\em coreflective} if the inclusion functor ${\mathrm{i}}\colon {\mathcal A}\to {\mathcal B}$ has a right adjoint. This adjoint is called a {\em coreflector}.

\begin{proposition}\label{prop:ce} The category ${\mathsf{ConRepr}}(S)$ is a coreflective subcategory of the category ${\mathsf{Repr}}(S)$. The coreflector is given by the functor $\Psi\Phi$.
\end{proposition}

\begin{proof} Let $(X,\mu)$ be a non-strict $S$-set.  Just as in the proof of Theorem \ref{th:equiv}, we have the map $\beta_{\mu}\colon  \Psi\Phi(X,\mu)\to X$ given by \eqref{eq:beta}.
 This map is surjective,  and is injective if and only if $\mu$ is connected. 
We show that the functor $\Psi\Phi$ is a right adjoint to the functor ${\mathrm{i}}\colon {\mathsf{ConRepr}}(S) \to {\mathsf{Repr}}(S)$ where the maps $\beta_{\mu}$ are the components of the counit $\beta\colon i\circ \Psi\Phi \to {\mathrm{id}}_{{\mathsf{Repr}}(S)}$.

Let  $(X,\mu)$ be any connected non-strict $S$-set, $(Y,\nu)$ be any non-strict $S$-set, and $$g\colon (X,\mu) \to (Y,\nu)$$ a morphism. To define the morphism $$f\colon (X,\mu) \to \Psi\Phi(Y,\nu),$$ let $x\in X$ and $e\in E(S)$ be such that $\mu(e)(x)$ is defined. Then it follows that $\nu(e)(f(x))$ is defined, as well. We set 
\begin{equation}\label{eq:def_f}
f(x)=[e,g(x)]\in \Psi\Phi(Y,\nu).
\end{equation}
For brevity, in this proof, we write $s\cdot x$ for $\mu(s)(x)$, $s\circ x$ for $\nu(s)(x)$ and $s*x$ for $\Psi\Phi(\nu)(s)(x)$.
Note that if $h\cdot x$ is defined where $h\in E(S)$ then $h\circ f(x)$ is defined, and an induction shows that $(e,g(x))\sim (h,g(x))$ follows from $(e,x)\sim (h,x)$, where the latter equivalence holds because $\mu$ is connected.
Therefore, the map $f$ is well-defined.
Let us show that $f$ is a morphism of non-strict $S$-sets. Assume that $s\cdot x$ is defined. This is equivalent to that ${\mathbf{d}}(s)\cdot x$ is defined. It follows that $s*[{\mathbf{d}}(s),g(x)]$ is defined as well, and applying \eqref{eq:action} we have
$$
s*[{\mathbf{d}}(s),g(x)]=[{\mathbf{r}}(s),s\circ g(x)]=[{\mathbf{r}}(s),g(s\cdot x)].
$$
On the other hand,  $f(s\cdot x)=[{\mathbf{r}}(s),g(s\cdot x)]$ holds by \eqref{eq:def_f}.  All that remains is to note that  the equality $g=\beta_{\nu} f$ is a direct consequence of the definitions of $f$ and $\beta_{\nu}$. 
\end{proof}

 \subsection{Transitive representations of $S$ as directed functors on $L(S)$} \label{sub:3.2}

\begin{proposition}\label{prop:trans}
The equivalence in Corollary \ref{cor:strict} restricts to an equivalence between the category of transitive $S$-sets and directed pullback preserving functors on $L(S)$.
\end{proposition}

\begin{proof} Let $(X,\mu)$, $(s,x)\mapsto s\cdot x$, if defined, be a transitive $S$-set. We show that the functor $\Phi(X,\mu)$ is directed. Let $(e,x)$ and $(f,y)$ be objects of the category of elements $\int_{L(S)}\Phi(X,\mu)$ of $\Phi(X,\mu)$ and $s\in S$ be such that  $s\cdot x=y$. We put $t=fse$ and observe that $t\cdot x=y$ and also ${\mathbf{d}}(t)\leq e$, ${\mathbf{r}}(t)\leq f$. Observe that $({\mathbf{d}}(t),x)$ is an object of the category $\int_{L(S)}\Phi(X,\mu)$. Since $e\cdot x=x$, the arrow $(e,{\mathbf{d}}(t))$ of $L(S)$ is an arrow from 
$({\mathbf{d}}(t),x)$ to $(e,x)$ of the category $\int_{L(S)}\Phi(X,\mu)$. Since $t\cdot x=y$,  the arrow $(f,t)$ is an arrow from $({\mathbf{d}}(t),x)$ to $(f,y)$ of the category $\int_{L(S)}\Phi(X,\mu)$. It follows that the functor $\Phi(X,\mu)$ is directed.

In the reverse direction, assume that the functor $\Phi(X,\mu)$ is directed and let $x,y\in X$. Let $e,f\in E(S)$ be such that $e\cdot x$ and $f\cdot y$ are defined (such $e$ and $f$ exist since $X$ is effective: for some $s$ we have that $s\cdot x$ is defined, but then ${\mathbf d}(s)\cdot x$ is defined, as well). Since the category $\int_{L(S)}\Phi(X,\mu)$ is directed, there are $z\in X$ and $g\in E(S)$ such that $g\cdot z$ is defined, and in the category $\int_{L(S)}\Phi(X,\mu)$ there are arrows $(e,x)\leftarrow (g,z)\to (f,y)$. The arrow $(g,z)\to (f,y)$ is by definition an arrow $(f,s)$ in $L(S)$ from $g$ to $f$ such that $s\cdot z=y$. Likewise, the arrow $(g,z)\to (e,x)$ is an arrow $(e,t)$ in $L(S)$ from $g$ to $e$ such that $t\cdot z=x$. It follows that $st^{-1}\cdot x=y$ which implies that the action is transitive.
\end{proof}

We now briefly recall the classical result due to Boris Schein \cite{Sch} (see also  \cite{H,LMS}) of the structure of transitive $S$-sets\footnote{We recall our convention that all $S$-sets are effective.}. 

An inverse subemigroup $H$ of $S$ is called {\em closed} if it is upward closed as a subset of $S$, i.e. $H^{\uparrow}=H$. Let $H$ be a closed inverse subsemigroup of $S$. A {\em coset} with respect to $H$ is a set $(xH)^{\uparrow}$ where ${\mathbf{d}}(x)\in H$. Let $X_H$ be the set of cosets with respect to $H$. Define the structure of an $S$-set on $X_H$ 
by putting $s\cdot (xH)^{\uparrow}$ is defined if and only if $(sxH)^{\uparrow}$ is a coset in which case
\begin{equation}\label{eq:structure} s\cdot (xH)^{\uparrow}=(sxH)^{\uparrow}.
\end{equation}
The obtained $S$-set $X_H$ is transitive and any transitive $S$-set is equivalent to one so constructed. 

\begin{remark} {\em It follows that Proposition \ref{prop:trans} provides a link, which was not previously explicitly mentioned in the literature, between closed inverse subsemigroups of $S$ and directed and pullback preserving functors on $L(S)$.}
\end{remark}

\subsection{Universal representations and filtered functors on $L(S)$}\label{sub:3.3}

Let $H$ be a closed inverse subsemigroup of $S$. Recall that a {\em filter} in a semilattice is an upward closed subset $F$ such that $a\wedge b\in F$ whenever $a,b\in F$. Since the meet in $E(S)$ coincides with the product of idempotents, it follows that $E(H)$ is a filter in $E(S)$. Since $H$ is closed,  $H\supseteq E(H)^{\uparrow}$ always holds. On the other hand, for any filter $F$ in $E(H)$ we have that $F^{\uparrow}$ is a closed inverse subsemigroup of $S$.

An $S$-set $(X,\mu)$ is called {\em universal} \cite{LMS}, if it is equivalent to a representation of $S$ on cosets with respect to a closed inverse subsemigroup $F^{\uparrow}$, where $F$ is a filter in $E(S)$. The following result  is mentioned without proof in \cite{LMS}. We provide a proof for completeness.

\begin{proposition}\label{prop:torsors} An $S$-set  $(X,\mu)$ is an $S$-torsor if and only if it is universal. 
\end{proposition}

\begin{proof} Let $(X,\mu)$, $(s,x)\mapsto s\cdot x$, where defined, be an $S$-set. Let $x\in X$ and put 
$$
H=\{s\in S\colon s\cdot x \text{ is defined and } s\cdot x=x\}.
$$
Then $H$ is a closed inverse subsemigroup of $S$, and $(X,\mu)$ is equivalent to the structure of an $S$-set, $(X_H,\nu)$, given in \eqref{eq:structure}, on the set $X_H$ of cosets with respect to $H$. We may thus assume that $(X,\mu)=(X_H,\nu)$. 

Assume that $(X_H,\nu)$ is an $S$-torsor. We show that $H=E(H)^{\uparrow}$. It is enough to verify that $H\subseteq E(H)^{\uparrow}$.
Let $s\in H$. Since $(X_H,\nu)$ is free, the equalities $$s\cdot x = {\mathbf d}(s)\cdot x=x$$
imply that there is some $c\leq s, {\mathbf d}(s)$ such that $c\cdot x=x$. Therefore $c\in E(H)$ and $s\geq c$, so that we have the inclusion $H\subseteq E(H)^{\uparrow}$.

Conversely, assume that $(X_H,\mu)$ is universal and let $s,t\in S$ and $x\in X_H$ be such that $s\cdot x=t\cdot x$.
Then there are some $e,f\in E(H)$ such that $s\geq e$, $t\geq f$ such that $e\cdot x$ and $f\cdot x$ are defined, and then of course $e\cdot x=f\cdot x=x$. We put $h=ef$. Then $s,t\geq h$ and $h\cdot x=x$, so that $(X_H,\mu)$ is an $S$-torsor.
\end{proof}

The following result follows from Proposition \ref{prop:torsors} and  \cite[Proposition 3.9]{FH} stated there without proof.  

\begin{proposition}\label{prop:ff} The equivalence in Proposition \ref{prop:trans} restricts to an equivalence between the category of universal $S$-sets and the category of filtered functors on $L(S)$. Consequently, the category of points of the topos ${\mathcal B}(S)$ is equivalent to the category of universal $S$-sets.
\end{proposition}

\begin{proof} Let $(X,\mu)$ be a universal $S$-set. Assume that we have two objects $(e,x)$ and $(f,y)$ and two arrows $$
(e,x) \drightarrow{(f,s)}{(f,t)} (f,y)
$$ in the category of elements $\int_{L(S)}\Phi(X,\mu)$. This implies that ${\mathbf d}(s)={\mathbf d}(t)=e$ and $s\cdot x=t\cdot x=y$. Since $(X,\mu)$ is free, there is $c\leq c,s$ such that $c\cdot x=y$.
Since $c\leq s$ we have that $c=sg$ for some $g\in E(S)$ where we may assume that $g\leq e$. Then
${\mathbf d}(c)=g$. This and $c\leq t$ yield that $c=tg$. It follows that there is an arrow
$$
(g,x)\stackrel{(e,g)}{\longrightarrow} (e,x)
$$
in the category $\int_{L(S)}\Phi(X,\mu)$. Since $(f,s)(e,g)=(f,c)=(f,t)(e,g)$ in $L(S)$, the diagram 
$$
(g,x) \stackrel{(e,g)}{\longrightarrow}(e,x) \drightarrow{(f,s)}{(f,t)} (f,y)
$$
is commutative. Therefore, the functor $\Phi(X,\mu)$ satisfies axiom (F3) from the definition of a filtered functor (see Subsection \ref{subs:2.5}). It satisfies (F1) and (F2) due to Proposition \ref{prop:trans}, since universal $S$-sets are transitive.

Conversely, let $(X,\mu)$ be an $S$-set such that the functor $\Phi(X,\mu)$ is filtered. Assume that $s,t\in S$ and $x\in X$ are such that $s\cdot x=t\cdot x$. Let $e={\mathbf d}(s){\mathbf d}(t)$ and $h={\mathbf r}(s){\mathbf r}(t)$. Then $hse\cdot x=hte\cdot x$ and also ${\mathbf d}(hse)={\mathbf d}(hte)$, ${\mathbf r}(hse)={\mathbf r}(hte)$. We put
$p={\mathbf d}(hse)$ and $q={\mathbf r}(hse)$.
It follows that in the category  $\int_{L(S)}\Phi(X,\mu)$ we have two parallel arrows
$$
(p,x)  \drightarrow{(q,hse)}{(q,hte)} (q,y).
$$
By axiom (F3), there is a commutative diagram
$$
(r,z) \stackrel{(p,a)}\longrightarrow (p,x)  \drightarrow{(q,hse)}{(q,hte)} (q,y)
$$
in the category  $\int_{L(S)}\Phi(X,\mu)$. This means that $hsea=htea$. Then $hse{\mathbf r}(a)=hte{\mathbf r}(a)$ and also 
$$hse{\mathbf r}(a)\cdot y =hte{\mathbf r}(a)\cdot y = z,$$
which proves that $(X,\mu)$ is free. Applying Proposition \ref{prop:trans} (noting that filtered functors preserve pullbacks) and Proposition \ref{prop:torsors}, we conclude that $(X,\mu)$ is an $S$-torsor.
\end{proof}

\section{Principal bundles over inverse semigroups}\label{sec:bundles}

In this section, we obtain an equivalence between the category of universal representations of an inverse semigroup on \'etale spaces over a topological space $X$ and the category of principal $L(S)$-bundles over $X$. This extends the well known result for groups \cite[VIII.1, VIII.2]{MM}, and also is an analogue of Proposition \ref{prop:ff}, if in the latter one replaces the topos of sets by the topos ${\mathsf{Sh}}(X)$.

The following definition is taken from \cite{M}. Let $X$ be a topological space and ${\mathcal C}$ a small category. A functor $E\colon {\mathcal C}\to {\mathsf{Sh}}(X)$ is called a ${\mathcal C}$-{\em bundle}. 
If $E\colon {\mathcal C}\to {\mathsf{Sh}}(X)$ is a ${\mathcal C}$-{\em bundle}, $\alpha\colon c\to d$ is an arrow in ${\mathcal C}$ and $y\in E(c)$,  we put
$$
\alpha\cdot y=E(\alpha)(y)\in E(d).
$$

A ${\mathcal C}$-bundle $E$ is called {\em principal}, if for each point $x\in X$ the following axioms are satisfied by the stalks $E(C)_x$:
\begin{enumerate}
\item[(PB1)] (non-empty) There is an object $c$ of $C$ such that $E(c)_x\neq\varnothing$;
\item[(PB2)] (transitive) For any $y\in E(c)_x$ and $z\in E(d)_x$, there are arrows $\alpha\colon b\to c$ and $\beta\colon b\to d$ for some object $b$ of $C$, and a point $w\in E(b)_x$, so that $\alpha\cdot w=y$ and $\beta\cdot w=z$.
\item[(PB3)] (free) For any two parallel arrows $\alpha,\beta\colon c\rightrightarrows d$ and any $y\in E(c)_x$, for which $\alpha\cdot y=\beta\cdot y$, there exists an arrow $\gamma\colon b\to c$ and a point $z\in E(b)_x$ so that $\alpha\gamma=\beta\gamma$ and $\gamma\cdot z=y$.
\end{enumerate}

Principal ${\mathcal C}$-bundles are known to coincide with filtered functors from ${\mathcal C}$ to the topos ${\mathsf{Sh}}(X)$. It is immediate that given a principal ${\mathcal C}$-bundle $E\colon {\mathcal C}\to {\mathsf{Sh}}(X)$ and $x\in X$, the induced restriction to stalk functor $E_x\colon {\mathcal C}\to {\mathsf{Sets}}$, $c\mapsto E(c)_x$, is a filtered functor. 

For two principal ${\mathcal C}$-bundles $E$ and $E'$, a {\em morphism} from $E$ to $E'$ is simply a natural transformation $\varphi\colon E\to E'$, that is, a collection of sheaf maps $\varphi_c\colon E(c)\to E'(c)$, where $c$ runs through objects of ${\mathcal C}$, such that for each arrow $\alpha\colon c\to d$ in ${\mathcal C}$ and each $y\in E(c)$ we have that $\varphi_d(\alpha\cdot y)=\alpha\cdot\varphi_c(y)$. 
We have therefore defined the category 
${\mathsf{Prin}}({\mathcal C}, X)$
of {\em principal bundles} over ${\mathcal C}$.

In the case where ${\mathcal C}$ is Loganathan's category $L(S)$ for an inverse semigroup $S$, we will call a principal bundle over $L(S)$ a  {\em principal bundle over} $S$ and the category ${\mathsf{Prin}}(L(S), X)$ the category of {\em principal bundles over} $S$. We will write ${\mathsf{Prin}}(S, X)$ for ${\mathsf{Prin}}(L(S), X)$.

We now define the notion of a {\em universal} $S$-{\em set} in the topos ${\mathsf{Sh}}(X)$.  Let $\pi \colon E\to X$ be an \'etale space and assume that a structure of an $S$-set $(E,\mu)$, $(s,x)\mapsto s\cdot x$, if defined, is given on $E$  such that the following conditions are met:
\begin{enumerate}
\item[(U1)] (effective on each stalk)  For any $x\in X$, there is at least one point $e\in E_x$ such that $s\cdot e$ is defined for some $s\in S$ (this in particular implies that all stalks are non-empty, that is, the map $\pi$ is surjective).
\item[(U2)] (domains are open) For any $s\in S$ the set $\{e\in E\colon s\cdot e \text{ is defined}\}$ is open.
\item[(U3)] (stalks are invariant)  For any $s\in S$ and $e\in E$, if $s\cdot e$ is defined then $\pi(s\cdot e)=\pi(e)$.
\item[(U4)] (universal on stalks) For any $s\in S$ and $x\in X$, $(E_x,\mu|_{S\times E_x})$ (which is well-defined by (U3)) is a universal $S$-set.
\item [(U5)] (continuous) The partially defined map $S\times E \to E$, $(s,x)\mapsto s\cdot x$, is continuous ($S$ is considered as a discrete space and $S\times E$ as a product space). 
\end{enumerate}

It is easy to see that (U4) implies (U1), so that (U1) may be omitted from the above list.

Let $\pi \colon E\to X$, $\pi'\colon E'\to X$ be \'etale spaces and $(E,\mu)$, $(s,e)\mapsto s\cdot x$, if defined,  $(E',\nu)$, $(s,e)\mapsto s\circ x$, if defined, be universal $S$-sets in the topos  ${\mathsf{Sh}}(X)$.  A morphism 
$$
f\colon (E,\mu) \to (E',\nu)
$$ 
is defined as a morphism $f\colon E\to E'$ of \'etale spaces (that is, a continuous map such that $\pi=\pi'f$, cf. \cite{MM})  which is simultaneously a morphism of $S$-sets (that is, if $s\cdot e$ is defined then $s\circ f(e)$ is defined and $f(s\cdot x)=s\circ f(x)$).
We denote the category of universal $S$-sets in the topos ${\mathsf{Sh}}(X)$ by ${\mathsf{Univ}}(S,X)$.

\begin{theorem}\label{th:sheaves} There is an equivalence of categories
$$
{\mathsf{Prin}}(S, X) \,\,
\drightleftarrow{\tau}{\rho} \,\, {\mathsf{Univ}}(S,X).
$$
\end{theorem}

\begin{proof}

We begin with the construction of the functor 
$$\tau\colon  {\mathsf{Prin}}(S, X)\to {\mathsf{Univ}}(S,X).$$
 Let $E\colon L(S) \to {\mathsf{Sh}}(X)$ be a principal bundle over $X$ and $x\in X$. 
 We first describe the colimit sheaf $\widetilde{E}\in {\mathsf{Sh}}(X)$. By definition, for each $x\in X$
we have a filtered functor $f_x\colon L(S)\to {\mathsf{Sets}}$ obtained by restricting $E$ to the stalks over $x$.
We now apply the functor $\Psi$ to each $f_x$.  Proposition \ref{prop:ff} ensures us that each $\Psi(f_x)$ is a universal $S$-set. Note that each of the sets $\Psi(f_x)$ is non-empty by (PB1).
As the stalks of the colimit sheaf are colimits of stalks, we may set 
$$\widetilde{E}=\bigcup_{x\in X} \Psi(f_x)$$
to be a disjoint union of all the sets  $\Psi(f_x)$.
 The projection map 
 $$p\colon \widetilde{E}\to X$$ 
 given by $y\mapsto x$ if $y\in \Psi(f_x)$.
 
We may identify each space $E(e)$ with its image under the inclusion into $\widetilde{E}$. The colimit topology on  $\widetilde{E}$ is the finest topology which makes all inclusion maps $E(e)\hookrightarrow \widetilde{E}$ continuous. The base of this topology is formed by the sets $A\subseteq \widetilde{E}$ such that $A\subseteq E(e)$ for some $e$. 
The structure of  an $S$-set, $(s,y)\mapsto s*y$, where defined, on $\widetilde{E}$ is induced by the structures of $S$-sets on each on $\Psi(f_x)$, given by \eqref{eq:action}. 

We now prove that for each $s\in S$ the set 
$$
D_s=\{y\in \widetilde{E}\colon s*y \text{ is defined}\}
$$
is open. Clearly, $D_s=D_{{\mathbf d}(s)}$. 
Let $y\in D_{{\mathbf d}(s)}$ and $A$ be a neighbourhood of $y$ in $\widetilde{E}$. Since the inclusion map $i\colon E_{{\mathbf{d}}(s)}\hookrightarrow \widetilde{E}$ is open, we have that $ii^{-1}(A)$ is a neighbourhood of $y$ in $\widetilde{E}$ which is contained in $D_{{\mathbf d}(s)}$. This implies that the set $D_{{\mathbf d}(s)}$ is open.
Using the fact that the translation maps are continuous, it is routine to verify that the partially defined map $S\times E \to E$, $(s,x)\mapsto s\cdot x$, is continuous.

We now define $\tau$ on morphisms. Let $E$ and $E'$ be principal $S$-bundles and $\varphi\colon E\to E'$ be a natural transformation. The family of continuous maps $E(e)\to E'(e)$ for each $e\in E(S)$ and the construction of $\widetilde{E}$ yield a continuous map $\tau(\varphi)\colon \widetilde{E}\to \widetilde{E'}$ which obviously satisfies the definition of a morphism of universal representations.

We now turn to the construction of the functor $\rho\colon  {\mathsf{Univ}}(S,X) \to {\mathsf{Prin}}(S, X)$.
Let $p\colon \widetilde{E}\to X$ be an \'etale space and $(\widetilde{E},\mu)$, $(e,x)\mapsto e*x$, if defined, a structure of an $S$-set on $\widetilde{E}$ which satisfies (U1) -- (U5).
We fix $e\in E(S)$ and let
$$
E(e)=\{y\in \widetilde{E}\colon e*y\text{ is defined}\}.
$$
Define $p_e\colon E(e)\to X$ to be the restriction of the map $p$ to $E(e)$. Clearly, $p_e$ is a local homeomorphism.
By Proposition \ref{prop:ff}, for each $x\in X$, the restriction of $*$ to $E_x$ gives rise to a filtered functor
$\Phi(\widetilde{E},\mu)_x\colon L(S)\to {\mathsf{Sets}}$ and it is routine to verify that these give rise to a filtered functor $\rho(p\colon \widetilde{E}\to X)\colon L(S)\to {\mathsf{Sh}}(X)$.

To define $\rho$ on morphisms, we observe that a morphism $\psi \colon E\to E'$ of universal representations yields a family of maps $E(e)\to E'(e)$ for each $e\in E(S)$. By construction, these maps are continuous and are components of a natural transformation from $\rho(E)$ to $\rho(E')$.
\end{proof}

It follows that $S$-torsors in the topos ${\mathsf{Sh}}(X)$ can be defined as universal $S$-bundles.

As a direct consequence of Theorem \ref{th:equiv} and an analogue of Theorem  \ref{th:filt} for the topos ${\mathsf{Sh}}(X)$, we obtain the following result.

\begin{corollary} The category of geometric morphisms ${\mathsf{Geom}}({\mathsf{Sh}}(X), {\mathcal{B}}(S))$ is equivalent to the category ${\mathsf{Univ}}(S,X)$ of universal $S$-bundles in the topos  ${\mathsf{Sh}}(X)$. 
\end{corollary}

\begin{example}[A universal bundle associated with the groupoid of filters of $S$]

{\em Let $S$ be an inverse semigroup. A {\em filter} in $S$ is a filter with respect to the natural partial order in $S$, that is, a nonempty subset $F$ of $S$ such that
\begin{enumerate}
\item[(F1)] $a\in F$ and $b\geq a$ imply that $b\in S$;
\item[(F2)] if $a,b\in F$ then there is $c\in F$ such that $c\leq a,b$.
\end{enumerate}

Let $E=E(S)$ and $\hat{E}$ denote the set of filters in $E$. A filter $F$ in $E$ defines a nonzero semilattice homomorphism, called a {\em semi-character}, $\varphi_F\colon E\to \{0,1\}$ such the inverse image of $1$ is $F$, and conversely, any nonzero semilattice homomorphism $\varphi\colon E\to \{0,1\}$ defines a filter $\varphi^{-1}(1)$  in $E$. These assignments are mutually inverse, so that the elements of $\hat{E}$ can be equivalently looked at as semi-characters. The space $\hat{E}$ is topologized as a subspace of the product space $\{0,1\}^E$ where $\{0,1\}$ is a discrete space. The space $\hat{E}$ is locally compact and is known as the {\em filter space} or the {\em semi-character space} of $S$.

Let ${\mathcal G}$ denote the set of filters in $S$. For $s\in S$ let 
$M(s)=\{F\in {\mathcal G}\colon s\in F\}.$
The set ${\mathcal G}$ is topologized by letting the sets $M(s)\cap M(s_1)^c \cap \dots \cap M(s_n)^c$, where $s,s_1,\dots s_n$, $n\geq 0$, to be a base of the topology. The connection between filters in $S$ and filters in $E$ was studied in detail in \cite{LMS}. If $F\in {\mathcal G}$, the assignment 
$$
\mathrm{d}(F)=\{\mathbf{d}(a)\colon a\in F\}\in \hat{E}
$$
defines a map $\mathrm{d}\colon {\mathcal G}\to \hat{E}$ which  is a local homeomorphism (in fact, this map is equivalent to the domain map of the {\em groupoid of filters} of $S$).

We have an action of $S$ on each stalk ${\mathcal G}_F$ of the \'etale space $({\mathcal G}, {\mathrm{d}}, X)$ which is just  the universal action of $S$ on the set of cosets with respect to the closed inverse subsemigroup $F^{\uparrow}$.
It is routine to verify that these actions define on ${\mathcal G}$ the structure of a universal $S$-bundle which is natural to call the {\em universal bundle associated to the domain map of the groupoid of filters of} $S$.}
\end{example}

\section{Towards actions of inverse semigroups in an arbitrary topos}\label{sec:fin}

In \cite[Definition 2.14]{FH}, Funk and Hofstra proposed a way to define a notion of a torsor for an arbitrary inverse semigroup $S$ in an arbitrary (Grothendieck) topos. Their definition \cite[Definition 2.14]{FH} is based on the concept of a semigroup $S$-set in an arbitrary topos: for an inverse semigroup $S$ they consider an internal semigroup $\Delta(S)$.
In the topos of sets, a semigroup $S$-set $X$ is  a (pre)homomorphism from $S$ to the partial transformation semigroup ${\mathcal{PT}}(X)$ on $X$. For $S$ inverse, only semigroup $S$-sets for which the action is by partial bijections should be considered. The diagrammatic definition of partial bijections is not written in \cite{FH}, but can be done. Omitting the requirement of partial bijections leads to an incorrect claim in \cite{FH}. 

Section~6 of \cite{FH} discusses the actions of inverse semigroups in an arbitrary topos.  However, the claim `If $T$ is inverse, then a $T$-set $T\to M(X)$ necessarily factors through $I(X)\subseteq M(X)$' is incorrect (where $I(X)$ is `the object of partial bijections'). We now provide an example that, for the topos of sets where the meaning of ${\mathcal I}(X)$ is clear (the symmetric inverse semigroup on $X$), shows the claim to be incorrect.

\begin{example}\label{ex:counterexample}{\em
Let $S$ be a linearly ordered set considered as a semilattice and let $|S|>1$. The map $\mu\colon S \to {\mathcal{PT}}(S)$ given by $x\mapsto \varphi_x$, where $\varphi_x(y)= x\wedge y$, $y\in S$ is a homomorphism, but this is not an inverse semigroup $S$-set as the action is not by partial bijections. It is also easy to check that $\mu$ is free and transitive according to \cite[Definition 2.8]{FH} and \cite[Definition~2.14]{FH}. 
}
\end{example}

If $S$ is an inverse semigroup, then the internal semigroup $\Delta(S)$ can be readily endowed with the structure of an `internal inverse semigroup' as $\Delta$ preserves the logic needed to express the fact of being an inverse semigroup (e.g. the varietal definition). Thus the approach taken in \cite{FH} of internalizing $S$ as a semigroup looks simpler than the possible  internalizing it as an inverse semigroup. We, however, believe, that it is more natural, e.g., from the perspective of  the cohomology theory \cite{Log}, to keep the idempotents of $S$ and to internalize ${\mathcal H}$-classes of $S$ and the connection between them. This is the approach we outline below.

Let ${\mathcal{E}}$ be a (Grothendieck) topos and  $S$ an inverse semigroup.  We define an action of an inverse semigroup in an arbitrary topos which arises from a functor $L(S)\to {\mathcal{E}}$. 
Then  classes of functors $L(S)\to {\mathcal{E}}$ (such as pullback preserving functors, torsion-free functors or filtered functors) can be connected with respective classes of actions of $S$ in ${\mathcal{E}}$. In particular, actions connected to filtered functors, can be naturally called $S$-torsors.

Let $e,f\in E(S)$ be such that $e\mathrel{\mathcal D} f$. By $H(e,f)$ be denote the ${\mathcal H}$-class of $S$ which consists of all $s\in S$ satisfying  ${\mathbf{d}}(s)=f$ and ${\mathbf{r}}(s)=e$. Note that any ${\mathcal H}$-class of $S$ is of the form $H(e,f)$ for some $e\mathrel{\mathcal D} f$. We can bring all the sets $H(e,f)$ up to ${\mathcal{E}}$ by considering their images $\Delta H(e,f)$ under the constant sheaf functor $\Delta$. 

Let $A\colon L(S)\to {\mathcal{E}}$ be a functor. The colimit construction, given in \cite{FH} for the topos of sets, extends to ${\mathcal{E}}$, and we construct the colimit object ${\mathcal X}$ of the
composite of the functors
$$
E(S)\to L(S)\stackrel{A}{\to} {\mathcal{E}}.
$$
In particular, all objects $A(e)$ are subobjects of ${\mathcal X}$. Therefore, a morphism $A(e)\to A(f)$ in ${\mathcal{E}}$ can be thought of as a `partial' morphism of ${\mathcal X}$. Note that if $S$ is a monoid with unit $1$ then 
we have ${\mathcal X}=A(1)$.
The restriction of $A$ to $E(S)$ gives us a functor $E(S)\to {\mathcal{E}}$. If $S$ is a group, this functor just selects an objects in ${\mathcal{E}}$, in particular, the object ${\mathcal X}$. Reasoning similarly as in \cite[p.~432]{MM}, we see that for each ${\mathcal H}$-class $H(e,f)$ the functor $A$ gives rise to a map
\begin{multline}\label{eq:manipulation}
H(e,f)\to {\mathrm{Hom}}_{{\mathcal{E}}}(A(f),A(e)) \simeq {\mathrm{Hom}}_{{\mathcal{E}}}(1, A(e)^{A(f)})\simeq \\ {\mathrm{Hom}}_{{\mathcal{E}}}(\Delta 1, A(e)^{A(f)})\simeq {\mathrm{Hom}}_{\mathsf{Sets}}(1, \Gamma(A(e)^{A(f)}))\simeq \Gamma(A(e)^{A(f)})
\end{multline}
(here $1$ denotes the terminal object of ${\mathcal E}$).
We obtain the map 
\begin{equation}
\label{eq:manip1}
\Delta H(e,f) \to A(e)^{A(f)},
\end{equation} and applying the adjunction between product and exponentiation,
\begin{equation}\label{eq:manip2}
\Delta H(e,f) \times A(f)\to A(e).
\end{equation} 

We recall that every morphism in $L(S)$ is a composition of some $({\mathbf{r}}(s),s)$ and some $(e,f)$, where $e,f\in E(S)$. Therefore, a functor $A\colon L(S)\to {\mathcal E}$ is determined by its restriction to $E(S)$ and by translations along isomorphisms $({\mathbf{r}}(s),s)$. The restriction of $A$ to $E(S)$ in the group case degenerates to selecting an object in ${\mathcal E}$, so it is natural to keep this restriction as a part of the definition of an $S$-set associated to $A$ in ${\mathcal E}$. The translations along isomorphisms are internalized  using \eqref{eq:manipulation}, \eqref{eq:manip1} and~\eqref{eq:manip2}.

\section*{Acknowledgements} We thank Andrej Bauer and Alex Simpson for useful discussions. We are also grateful to  Jonathon Funk and Pieter Hofstra for helpful communication, to the referee for their comments, as well as the editor for facilitating a fruitful discussion.

\end{document}